\documentclass[11pt]{amsart}
\usepackage{amsfonts}
\usepackage{amssymb}
\usepackage{amscd}
\input xy
\xyoption{all}
\title[Tilting in dimension two]{Filtrations in abelian categories with a tilting object of homological dimension two}
\author{Bernt Tore Jensen, Dag Madsen and Xiuping Su}

\newtheorem{theorem}{Theorem}

\newtheorem{lemma}[theorem]{Lemma}
\newtheorem{proposition}[theorem]{Proposition}

\newtheorem{corollary}[theorem]{Corollary}

\newcommand{\lra}{\longrightarrow}

\newcommand{\ra}{\rightarrow}
\newcommand{\sdp}{\times\kern-.2em\vrule height1.1ex depth-.05ex}
\newcommand{\epi}{\lra \kern-.8em\ra}
\newcommand{\catA}{\mathcal{A}}
\newcommand{\catB}{mod B}
\newcommand{\derA}{\mathcal{D}^b(\mathcal{A})}
\newcommand{\derB}{\mathcal{D}^b(B)}
\newcommand{\G}[1]{\mathcal{G}_{#1}}
\newcommand{\F}[1]{\mathcal{F}^{#1}}
\newcommand{\K}[1]{\mathcal{K}^{#1}}
\newcommand{\E}[1]{\mathcal{E}^{#1}}

\begin{document}

\begin{abstract}
We consider filtrations of objects in an abelian category $\catA$ induced by
a tilting object $T$ of homological dimension at most two. We define three
disjoint subcategories with no maps between them in one direction, such that
each object has a unique filtation with factors in these categories. This
filtration coincides with the the classical two-step filtration induced by torsion pairs in dimension one.
We also give a refined filtration, using the derived equivalence between
the derived categories of $\catA$ and the module category of $End_\catA (T)^{op}$. The
factors of this filtration consist of kernel and cokernels of maps between objects
which are quasi-isomorphic to shifts of $End_\catA (T)^{op}$-modules via the derived
equivalence $\mathbb{R}Hom_\catA(T,-)$.
\end{abstract}
\maketitle

\section*{Introduction}

Let $k$ be a field and let $\catA$ be either a noetherian abelian $k$-category
with finite homological dimension and hom-finite bounded derived category
$\derA$, or let $\catA$ be the module category of a finite dimensional algebra
$A$. In the first case we will assume that there is a locally noetherian abelian
Grothendieck $k$-category $\catA'$ with finite homological dimension such that
$\catA\subseteq \catA'$ is the subcategory of noetherian objects. For more
information about tilting in noetherian categories, please see for example Baer
\cite{Baer} and Bondal \cite{Bondal}.

Let $T\in \catA$ be a tilting object. That is, $T\in \catA$ is an object without
self-extensions in non-zero degrees and which induces a pair of mutually inverse
derived equivalences $$F=\mathbb{R}Hom(T,-):\derA \lra \derB \mbox{ and }
G:\derB\lra \derA,$$ where $B=End(T)^{op}$ is a finite dimensional algebra
and $\derB=\mathcal{D}^b(modB)$ is the bounded derived category of finite dimensional
left $B$-modules.

Let $\tau_TM$ denote the trace of $T$ in $M$. If the projective dimension
of $T$ is one (i.e. $Ext^i_\catA(T,-)=0$ for $i>1$), then there is a canonical
short exact sequence $$0 \lra \tau_TM \lra M \lra M/\tau_TM \lra 0,  \;    \;
\;  \;  \;  \;  \;  (1)$$ where $\tau_TM\in FacT$, $M/\tau_TM\in RejT$,  $FacT$
is the torsion subcategory of objects $N$ with a surjective homomorphism
$T^n\rightarrow N$ for some $n>0$, and $RejT$ is the torsion free
subcategory of objects $N$ with $Hom(T,N)=0$. It is well known that
the sequence (1) generalises when $T$ has higher homological dimension
to a torsion theory in the derived category,
but our interest is in constructing filtrations in the abelian category. Our aim is to generalize the
sequence (1) to abelian categories with a tilting object with higher
homological dimension. As a first step towards that goal, we present in this paper the
solution for homological dimension two, that is, to abelian categories $\catA$ with a
tilting object $T$ with $H^iFX=Ext^i_\catA(T,X)=0$ for all $X\in \catA$ and $i\geq 3$.
The sequence will be replaced by a filtration with three terms,  and will
coincide with the sequence (1) when the projective dimension of $T$ is one.

Let $F^i=H^iF=Ext^i_\catA(T,-)$ and let $G_i=H_iG$. Let $\G i$ denote
the subcategory of $\catB$ consisting of modules $X$ with $G_jX=0$ for
$i\neq j$. Similarly, we define $\F i$ to be the subcategory of $\catA$
consisting of objects $X$ with $F^jX=0$ for $i\neq j$.
It is clear that the categories $\F i$ are closed under extensions and pairwise disjoint,
and there are induced equivalences $F_{|\F i}:\F i \lra \G i[-i]$ and $G_{|\G i}:\G i \lra \F i[i]$,
where $[i]$ denotes the usual shift in the derived category. In other
words, $$\bigcup \F i\subseteq \catA$$ are objects in $\catA$ which are
$B$-modules via the derived equivalence. In dimension one we have
$FacT=\F 0$ and $RejT=\F 1$, and the short exact sequence (1)
shows how to canonically reconstruct $\catA$ using the subcategories $\F i$.
A natural thing to try in higher homological dimensions is to consider
filtrations with subfactors in the disjoint subcategories $\F i$, see Tonolo \cite{Tonolo}
for this approach. Unfortunately, such filtrations fail to exist in general,
even in dimension two, and we give an example in Section \ref{ext}. The problem
is that the categories $\F i$ are two small to filter any object in $\catA$.
To ensure that any object can be filtered we enlarge the categories $\F i$
to categories $\E i$, which are still closed under extensions and pairwise disjoint,
and are large enough to allow a unique filtration for any object in $\catA$.
Moreover, the categories $\E i$ are explicitly constructed, using the categories
$\F i$.

Let $\K 0$ be the full subcategory of objects which are cokernels of
monomorphisms from objects in $\F 2$ to objects in $\F 0$,  let
$\K 2$ be the full subcategory of objects which are kernels of
epimorphisms from objects in $\F 2$ to objects in $\F 0$, and let
$\K 1 = \F 1$. Note that, $\F 0\subseteq \K 0 \subseteq KerF^2$ and
$\F 2 \subseteq \K 2\subseteq KerF^0$.  Let $\E i$ be the extension
closure of $\K i$, that is, $\E i\subseteq \catA$ is the smallest
subcategory closed under extensions, and containing $\K i$. Note
that $\E 1 = \K 1 = \F 1$. The categories $\E i$ have the following
key properties.

\begin{theorem} \label{extra} If $j>i$, then $Hom(\E i,\E j)=0$. In particular, the
subcategories $\E 0$, $\E 1$ and $\E 2$ are pairwise disjoint.
\end{theorem}

In abelian categories with tilting object of
homological dimension two, we have the following filtration,
generalizing the short exact sequence (1).

\begin{theorem} \label{Theo1}
Let $T\in \catA$ be a tilting object with homological dimension at most two
and let $X\in \catA$. Then there is a unique and functorial filtration $0=X_0 \subseteq
X_1 \subseteq X_2 \subseteq X_3 = X$ with $X_{i+1}/X_{i}\in \E i$ for $i=0,1,2$.
\end{theorem}

If the homological dimension of $T$ is one, then $\E 0=\F 0$ and $\E 2=0$,
and so we recover the sequence (1) as a corollary.

\begin{corollary}
Let $T\in \catA$ be a tilting object with homological dimension one and let
$X\in \catA$. Then there is a canonical short exact sequence $0 \lra X_1
\lra X \lra X/X_1 \lra 0$ with $X_1\in \E 0$ and $X/X_1\in \E 1$.
\end{corollary}

We also prove the existence of the following refined filtration.
The proof is constructive and allows us to compute the filtration
in concrete examples.

\begin{theorem} \label{Theo2} Let $T\in \catA$ be a tilting module with
homological dimension two and let $X\in \catA$. Then there is a
filtration $(0) = Z_0 \subseteq ... \subseteq Z_n \subseteq Y_n \subseteq ...
\subseteq Y_0 = X$ with $Y_i/Y_{i+1}\in \K 2$, $Z_{i+1}/Z_i \in \K 0$ and
$Y_n/Z_n\in \K 1$.
\end{theorem}

By the uniqueness in Theorem \ref{Theo1}, we remark that
$X_1=Z_n$ and $X_2=Y$. Also, the theorem and its proof show how the
the objects in $\E i$ are constructed from the categories $\F i$.
An important consequence is an explicit and
canonical reconstruction of $\catA$ using the subcategories $\F i$.

Examples of tilting objects $T$ satisfying the hypothesis in the theorems
above include the $k$-dual $DA$ for an algebra of global dimension two.
Another class of examples are tilting sheaves over smooth projective
surfaces, for example over the projective plane.

The remainder of this paper  is organized as follows.  In Section 1 we
analyse the equivalence $GF\cong 1_\catA$. The proof of Theorem \ref{Theo2}
is given in Section 2, and the proof of
Theorem \ref{Theo1} is given in Section 3. In Section 4 we discuss the
extension closure of $FacT$ where $T$ is a tilting module in a
category $\catA$ of modules of a finite dimensional
algebra.

\section{Some homological algebra}

Let $T\in \catA$ be a tilting object. We will assume that $T$ has homological
dimension at most two, that is $F^iX=Ext_\catA^{i}(T,X)=0$ for all $X\in \catA$
and $i\geq 3$. In this section we prove some general homological properties for
tilting objects with homological dimension two.

\begin{lemma} \label{prevlemma}
If $T$ has homological dimension two, then $G_iM=0$ for all $M\in modB$ and $i\geq 3$.
\end{lemma}
\begin{proof}
Let $M\in mod B$. There is an exact triangle $$\tau_{\geq 3} GM
\longrightarrow GM \longrightarrow \tau_{<3}GM \longrightarrow \tau_{\geq 3}
GM[1],$$ where $\tau$ denotes truncation.  If we apply $F$ we get the triangle
$$F\tau_{\geq 3} GM \longrightarrow M \longrightarrow F\tau_{<3} GM
\longrightarrow F\tau_{\geq 3} GM[1]$$ Since $T$ has homological dimension
$2$, we see that $H^iF\tau_{\geq 3} GM=0$ for $i\geq 0$. So the map
$M\longrightarrow  F\tau_{<3} GM$ induces isomorphisms in cohomologies, and
is therefore a quasi-isomorphism. This shows that $F\tau_{\geq 3} GM=0$ and therefore
$\tau_{\geq 3} GM=0$. Hence $G_jM=0$ for all $j \geq 3$.
\end{proof}

We let $J^i_j=G_jF^i$. Using the quasi-isomorphism $GF(X)\cong X$ for
$X\in \catA$ we get the following double complex.

\begin{lemma} \label{complexlemma}
\label{double} Let $X\in \catA$. Then $X$ is quasi-isomorphic to the total complex
of a double complex $$\xymatrix{& & & & \cdots \ar[r] & T^2_2 \ar[r] \ar[d]
\ar@{=}[dl] & T^2_1 \ar[r] & T^2_0 \ar[r] & 0 \\ & & \cdots \ar[r] &
T^1_2 \ar[r] \ar[d] & T^1_1 \ar[r] \ar@{=}[dl] & T^1_0 \ar[r] & 0 & &
&  \\ \cdots \ar[r] & T^0_2 \ar[r] & T^0_1 \ar[r] & T^0_0 \ar[r] & 0
& & & & }$$ where the horizontal homology is $J^i_jX$ for all $0 \leq i \leq 2$ and $j\geq 0$,
the double lines indicate total degree $0$, and $T^i_j$ is a finite direct sum of
summands of $T$.
\end{lemma}
\begin{proof}
The complex $FX$ has cohomology in degrees $0$, $1$ and $2$.
By taking projective resolutions of the cohomology, we construct
a double complex
$$\xymatrix{
& & & & \cdots \ar[r] & P^2_2 \ar[r] \ar[d] \ar@{=}[dl] & P^2_1 \ar[r] &
P^2_0 \ar[r] & 0 \\ & & \cdots \ar[r] & P^1_2 \ar[r] \ar[d] & P^1_1
\ar[r] \ar@{=}[dl] & P^1_0 \ar[r] & 0 & & &  \\ \cdots \ar[r] & P^0_2
\ar[r] & P^0_1 \ar[r] & P^0_0 \ar[r] & 0 & & & & }$$ with horizontal
cohomology $H^*FX$ and total complex quasi-isomorphic to $FX$. Then
the lemma follows by applying the functor $G_0$ and using the
isomorphism $GHom_\catA(T,T)=G_0Hom_\catA(T,T)\cong T$.
\end{proof}

We show that the horizontal homology of the double complex of Lemma
\ref{complexlemma} vanishes almost everywhere.

\begin{lemma}
We have $J^i_jX=0$ for
\begin{itemize}
\item[a)] $j>2$,
\item[b)] $i=0, j=1,2$, and
\item[c)] $i=2,j=0,1$.
\end{itemize}
\end{lemma}
\begin{proof}
Part $a)$ follows since $G_i=0$ for $i\geq 3$, by Lemma \ref{prevlemma}.

Let $i=0$ and $j>0$.  Let $Y$ be the total complex of the two top rows of
the double complex in Lemma \ref{complexlemma},
let $Z$ be the bottom row, let $W$ be the middle row,
and let $V$ be the top row. We have a triangle
$$Z \lra X \lra Y \lra Z[1]$$ and therefore $H_{j+1}Y\cong H_jZ$ since
$H_jX=0$ for $j\neq 0$. Moreover, there is the triangle $$W \lra Y \lra V \lra
W[1]$$ and therefore an exact sequence $$H_{j+1}W\longrightarrow
H_{j+1}Y \longrightarrow H_{j+1}V.$$ By Part $a)$ we have
$H_{j+1}W=J^j_{j+2}X=0=J^0_{j+3}X=H_{j+1}V$ and therefore
$H_{j+1}Y \cong H_j Z=J^0_jX=0$. This proves Part $b)$.

Part $c)$ is similar and is left to the reader.
\end{proof}

In the cases that the horizontal homology does not vanish, we get
the following edge effect.

\begin{lemma} \label{seqlemma}
We have a complex $$0 \lra J^1_2X \lra J^0_0X \lra X \lra J^2_2X
\lra J^1_0X \lra 0,$$ with homology at $X$ equal to $J^1_1X$, and
vanishing homology elsewhere.
\end{lemma}
\begin{proof}
Let $Y,Z,V$ and $W$ be as in the proof of the previous lemma.
By computing long exact sequence for the triangle $$Z \lra X \lra Y \lra Z[1]$$
we get an exact sequence $$0 \lra H_1Y \lra H_0Z \lra X \lra H_0Y \lra 0,$$
and also $H_{-1}Y=0$.
By computing long exact sequence for the triangle $$W \lra Y \lra V \lra W[1]$$
we get an exact sequence $$0 \lra H_0W \lra H_0Y \lra  H_0V \lra H_{-1}W \lra 0,$$
and an isomorphism $H_1 W \cong H_1 Y$.
By splicing these sequences together we get the complex $$0 \lra H_1Y \lra H_0Z
\lra X \lra H_0V \lra H_{-1}W \lra 0$$ with homology at $X$ equal to $H_0W$
and vanishing homology elsewhere. The lemma follows since $H_1Y=J^1_2X$,
$H_0Z=J^0_0X$, $H_0W=J^1_1X$, $H_0V=J^2_2X$ and $H_{-1}W=J^1_0X$.
\end{proof}

We note the following consequence.

\begin{lemma} \label{nextlemma} Let $X\in \catA$.
\begin{itemize}
\item[a)] If $X\in Ker F^0$ then there is an exact sequence $0 \lra J^1_1X \lra X
\lra J^2_2X \lra J^1_0X \lra 0.$
\item[b)] Any $X\in  Ker F^1$ decomposes uniquely as $X=J^0_0X\oplus
J^2_2X$.
\item[c)] If $X\in Ker F^2$ then  there is an exact sequence $0 \lra J^1_2X \lra
J^0_0X \lra X \lra J^1_1X \lra 0$.
\end{itemize}
\end{lemma}
\begin{proof}
If $X\in Ker F^0$ then $J^0_0X=0$ and the exact sequence in Part $a)$ is a special
case of the exact sequence in Lemma \ref{seqlemma}. Part $c)$ is similar.
If $X\in Ker F^1$, then the middle row of the double complex in Lemma \ref{complexlemma}
vanishes, and the total complex of the double complex decomposes into
the direct sum of the top and bottom row. So by taking homology, we see that
$X=J^0_0X\oplus J^2_2X$.
\end{proof}

We point out that the analysis we have done in this section is considerably more
complicated for higher homological dimension.  Also, several of the proofs can be done
quite efficiently using spectral sequences, we have however chosen more conceptual
arguments using elementary properties of derived categories.
For a systematic study of tilting using double complexes and spectral
sequences, see Keller-Vossieck \cite{Keller} and Butler \cite{Butler}.

\section{Proof of Theorem \ref{Theo2}}

We start by computing the subcategories $\F i$ and $\G i$ for $i\neq 1$.

\begin{lemma} \label{image}
For $i\neq 1$, the image of the functor $F^i:\catA \lra modB$ is dense in $\G i$.
Similarly, the image of the functor $G_i:\catB \lra \catA$ is dense in $\F i$.
\end{lemma}
\begin{proof}
We prove the lemma for $F^0$. The proofs for the
other functors are similar, and are left to the reader.

Since the homology in degree $1$ and $2$ in the total complex of the
double complex in Lemma \ref{double} is zero, we see that
$GF^0X\cong G_0F^0X$ for any $X\in \catA$, and so $GF^0X$ can
only have homology in degree $0$. Therefore $F^0X\in \G 0$.
Similarly, we have
$G_0M\in \F 0$ for a $B$-module $M$. Let $M\in \G 0$. Then
$M\cong FGM\cong FG_0M \cong F^0G_0M$, which shows that the
image of $F^0$ is dense in $\G 0$.
\end{proof}

The following lemma is an easy application of long exact sequence
in homology.

\begin{lemma} \label{zerolemma}
If $X\in \K 0$ is obtained as a cokernel of an injection $$0 \lra J_2 \lra
J_0 \lra X \lra 0,$$ $J_2\in \F 2, J_0\in \F 0$, then $F^0X\cong F^0J_0$,
$F^1X\cong F^2J_2$ and $F^2X=0$. Similarly, if $Y\in \K 2$ is obtained as a kernel
of a surjection $$0\lra Y \lra L_2 \lra L_0\lra 0,$$ $L_2\in \F 2, L_0\in \F 0$,
then $F^0Y=0$,  $F^1Y\cong F^0L_0$ and $F^2Y\cong F^2L_2$.
\end{lemma}

We are ready to prove the main inductive step in the construction
of the filtration.

\begin{lemma} \label{induction}
Let $X\in \catA$. We have a filtration $Z \subseteq Y \subseteq
X$ with $Z\in \K 0$, $X/Y\in \K 2$ and $Y/Z\cong J^1_1X$.
\end{lemma}
\begin{proof}
Let $Z$ be the image of the map $J^0_0X \lra X$, and $Y$ the kernel
of the map $X \lra J^2_2X$ in the complex in Lemma \ref{seqlemma}.
Then from  Lemma \ref{seqlemma} we see that $Y/Z\cong J^1_1X$.
Furthermore, $Z\in \K 0$, $X/Y\in \K 2$ by Lemma \ref{image}.
\end{proof}

For $X\in \catA$ let $d(X)$ be the dimension of $F^1X$. We have
$d(X)=d(J^1_1X)$ if $X\in \F 1$.

\begin{lemma} \label{desclemma}
\label{smaller} Let $X\in\catA$. Then $d(J^1_1X)\leq d(X)$. Moreover, if
$d(J^1_1X)=d(X)$ then $J^1_1X\in \F 1$.
\end{lemma}
\begin{proof}
Let $X\not\in \F 1$ with $d(X)>0$. Let $Z \subseteq Y \subseteq X$ be the
canonical filtration of Lemma \ref{induction}. There are long exact
sequences in homology $$F^1Z\lra F^1Y \lra F^1(Y/Z) \lra 0$$ and $$0
\lra F^1Y \lra F^1X \lra F^1(X/Y)$$ which shows that $d(J^1_1X)=
d(Y/Z)\leq d(Y) \leq d(X)$.

Now assume that $d(J^1_1X)= d(X)$. Then from the above long exact
sequences we see that $F^1Z\lra F^1Y$ is the zero map. Since
the inclusion $Z\subseteq X$ factors as $Z\subseteq Y \subseteq X$,
we also have that $F^1Z\lra F^1X$ is zero. So there is the long exact
sequence $$0\lra F^0Z \lra F^0X \lra F^0(X/Z) \lra F^1Z \lra 0$$
Now $F^0Z\cong F^0J^0_0X$ by Lemma \ref{zerolemma} and
$F^0J^0_0X\cong F^0X$ by the essential surjectivity of $F^0$. Hence
$F^1Z\cong F^0(X/Y)\in \G 0$. But we also have
$F^1Z\cong F^2J^1_2X\in \G 2$, by Lemma \ref{zerolemma}. But then
$F^1Z\in  \G 0 \cap \G 2 = 0$, and so $F^1Z=0$, and therefore
$F^2J^1_2X=0$. Hence $J^1_2X=0$ since $J^1_2X\in \F 2$.
Similarly, $J^1_0X=0$. Therefore $GF^1X$ has homology $J^1_1X$
concentrated in degree $1$, and so $J^1_1X \in \F 1$.
\end{proof}

We can now prove the second theorem stated in the introduction.

\begin{theorem} \label{Theorem1} Let $T\in \catA$ be a tilting object
with projective dimension two. For any $X\in \catA$ there is a
filtration $(0) = Z_0 \subseteq ... \subseteq Z_n \subseteq Y_n
\subseteq ... \subseteq Y_0 = X$ with $Y_i/Y_{i+1}\in \K 2$, $Z_{i+1}/Z_i
\in \K 0$ and $Y_n/Z_n\in \K 1$.
\end{theorem}
\begin{proof}
Let $Z_1=Z$ and $Y_1=Y$ be given by Lemma \ref{induction}. Given
$Z_i$ and $Y_i$, with $Y_i/Z_i\not \in \F 1$ we construct $Z_{i+1}$
and $Y_{i+1}$ by applying Lemma \ref{induction} to $Y_i/Z_i$, and obtain
the filtration $$(0) = Z_0 \subseteq ... \subseteq Z_{i+1} \subseteq Y_{i+1}
\subseteq ... \subseteq Y_0 = X$$ with $Y_i/Y_{i+1}\in \K 2$, $Z_{i+1}/Z_i
\in \K 0$, and $Y_{i+1}/Z_{i+1}=J^1_1Y_i/Z_i$. By Lemma \ref{smaller}, this
procedure must eventually stop.
\end{proof}

As a consequence of the construction we see that $Y_t/Z_t\cong
(J^1_1)^tX$ and that for any $X$, there exists a smallest integer
$t>0$ such that $(J^1_1)^tX=(J^1_1)^{t+1}X$. It would be
interesting to find a method to compute this number in general.

We end this section by showing the existence of the filtration in Theorem \ref{Theo1}.
A different construction, which also shows uniqueness and functoriality, will be given in
the next section.

\begin{corollary} \label{cor}
There is a filtration $0=X_0 \subseteq
X_1 \subseteq X_2 \subseteq X_3 = X$ with $X_{i+1}/X_{i}\in \E i$
for $i=0,1,2$.
\end{corollary}
\begin{proof}
With the notation of the previous theorem, let $X_1 = Z_n, X_2 = Y_n$
and $X_3=X$. Then $X_3/X_2$ is in the extension closure of $\K 2$,
$X_2/X_1\in \K 1 = \E 1$, and $X_1/X_0$ is in the extension closure of
$\K 0$. The proof follows.
\end{proof}

We will give a different construction of this filtration in the next
section.

\section{Proof of Theorem \ref{Theo1}}

We will now give an alternative description of the filtration from Corollary
\ref{cor}, which will show that the filtration is functorial, and so Theorem
\ref{Theo1} follows.

We show that $\K 0=Fac T$. That is, $\E 0$ consists of objects having
filtrations with subfactors from $Fac T$.

\begin{lemma} \label{FacT}
$\K 0=Fac T$
\end{lemma}
\begin{proof}
Assume that $X\in Fac T$. Then there is an approximation $$0 \lra
Y \lra T^n \lra X \lra 0$$ with $Y\in Ker F^1$. By Lemma
\ref{nextlemma} and Lemma \ref{image}, $Y$ decomposes as
$Y=Y_0\oplus Y_2$ where $Y_0\in \F 0$ and $Y_2\in \F 2$. The pushout
of the projection $Y\lra Y_2$
$$\xymatrix{ &0 \ar[d] & 0 \ar[d] && \\ & Y_ 0 \ar[d] \ar@{=}[r]& Y_0 \ar[d]&&
\\ 0 \ar[r] & Y\ar[r] \ar[d] & T^n \ar[d] \ar[r] & X \ar@{=}[d] \ar[r] & 0 \\
0 \ar[r] & Y_2 \ar[d] \ar[r] & T'\ar[d] \ar[r] & X\ar[r] & 0 \\ &0&0&&}$$
gives us a short exact sequence $$0 \lra Y_2 \lra T' \lra X \lra 0$$
with $T'\in \F 0$ which shows that $X\in \K 0$.

Conversely, assume that $X\in \K 0$. Then there is an epimorphism
$T'\lra X$ with $T'\in \F 0$.  The proof is complete if we can show that
$T'\in Fac T$. Let $T^n\lra T'$ be an approximation. Then
$Hom_\catA(T,T^n)\lra Hom_\catA(T,T')$  is surjective, and so
$$T^n\cong G_0Hom_\catA(T,T^n) \lra G_0Hom_\catA(T,T') \cong T'$$
is surjective, since $G_0$ is right exact, and the lemma follows.
\end{proof}

\begin{lemma}
$\E 0$ is closed under factors.
\end{lemma}
\begin{proof}
Any quotient of an object filtered in $Fac T$ is also filtered in $Fac T$, and so
$\E 0$ is closed under factors.
\end{proof}

We consider the subcategories $Ker F^i$. First note that $Ker F^i$ is
closed under extensions, $Ker F^0$ is closed under subobjects and
$Ker F^2$ is closed under factors. For a subcategory $\mathcal{C}\subseteq
\catA$ and $X\in \catA$, let $\tau_{\mathcal{C}}X$ denote the trace of
$\mathcal{C}$ in $X$.

\begin{lemma}  \label{keylemma} The following hold.
\begin{itemize}
\item[a)] For $X\in Ker F^0$ there is a canonical exact sequence
$0 \lra X'' \lra X \lra X' \lra 0$ with $X''=\tau_{\E 1}X\in \E 1$ and $X'\in \E 2$.
\item[b)] For $X\in Ker F^2$ there is a canonical  exact sequence
$0 \lra X'' \lra X \lra X' \lra 0$ with $X''=\tau_{\E 0}X\in \E 0$ and $X'\in \E 1$.
\end{itemize}
\end{lemma}
\begin{proof}
Let $X\in Ker F^0$. Then $\tau_{\E 1}X\in Ker F^0$, since $X\in Ker F^0$,
and $\tau_{\E 1}X \in Ker F^2$ since there is a surjection $E\lra \tau_{\E 1}X$
with $E\in \E 1\subseteq Ker F^2$. This shows that $\tau_{\E 1}X\in \E 1$.
There is a short exact sequence $$0 \lra \tau_{\E 1}X \lra X \lra Z \lra 0,$$
where $Z=X/\tau_{\E 1}X$. By taking pullback $$\xymatrix{0 \ar[r] &
\tau_{\E 1}X \ar[r] \ar@{=}[d] &  E' \ar[d] \ar[r] & E \ar[d]\ar[r] & 0 \\ 0 \ar[r]
& \tau_{\E 1}X \ar[r] &  X \ar[r] & Z \ar[r] &  0}$$ along any nonzero map
$E\lra Z$ from $\E 1$ to $Z$ and using that $\E 1$ is closed under extensions
we see that $\tau_{\E 1}Z=0$. We now show that $Z\in Ker F^0$. There is an
exact sequence $$0 \lra F^0Z \lra F^1\tau_{\E 1}X \lra Z' \lra 0,$$ where $Z'$
is the cokernel of the inclusion $F^0Z \lra F^1\tau_{\E 1}X$. Using the functor $G$ we get a
long exact sequence with zero terms except for $$0 \lra \tau_{\E 1}X \lra  G_1Z'
\lra J^0_0Z\lra 0.$$ In particular, $G_iZ'=0$ for $i\neq 1$ and so $G_1Z'\in \E 1$.
If $J^0_0Z$ is nonzero, then we have a nonzero map $G_1Z'\rightarrow J^0_0Z
\rightarrow Z$ contradicting that $\tau_{\E 1}Z=0$. Hence $J^0_0Z=0$ and so
$F^0Z=0$ since $F^0Z\in \G 0$. This proves that $Z\in Ker F^0$.

Let $Y_0=Z$ and let $Y_n=J^1_1Y_{n-1}$ for all $n>0$. By Lemma \ref{nextlemma}
there is an exact sequence $$0 \lra J^1_1Y_n \lra Y_{n} \lra J^2_2Y_{n}
\lra J^1_0Y_{n} \lra 0,$$ and so by Lemma \ref{desclemma} there exists a $t>0$
such that $Y_t=Y_{t+1}\in \E 1$. But then $Y_t=0$ since $Y_t\subseteq Z$
and $\tau_{\E 1}Z=0$. Then $Y_{t-1}\in \K 2$, since $J^2_2Y_{t-1}\in \F 2$
and $J^1_0Y_{t-1}\in \F 0$, by Lemma \ref{image}. Then by induction on $t$
we see that $Z=Y_0$ has a filtration with factors in $\K 2$, and therefore
$Z\in \E 2$. This completes the proof of Part a).

Now let $X\in Ker F^2$. Since $\E 0$ is closed under factors we
see that $\tau_{\E 0}X\in \E 0$. There is a short exact sequence
$$0 \lra \tau_{\E 0}X \lra X \lra Z \lra 0,$$ where $Z$ is the cokernel of
the inclusion $\tau_{\E 0}X \lra X $. Then $Z\in Ker F^2$, since $X\in
Ker F^2$ and $Ker F^2$ is closed under factors. By taking  pullback
$$\xymatrix{0 \ar[r] & \tau_{\E 0}X \ar[r] \ar@{=}[d] & T'\ar[d] \ar[r] &
T\ar[r] \ar[d] & 0 \\ 0 \ar[r] & \tau_{\E 0}X \ar[r] & X \ar[r] & Z \ar[r] & 0 }$$
along a non-zero map $T\lra Z$ we get a map from $\E 0$ to $X$ which
does not factor through the inclusion $\tau_{\E 0}X \lra X$, which is a
contradiction, and so $Z\in Ker F^0$. Therefore $Z\in \E 1=Ker F^0\cap Ker F^2$.
This completes the proof of Part b).
\end{proof}

For $X\in \catA$, let $X_0=0$, $X_1=\tau_{\E 0}X$ and let $X_2\subseteq X$ be the preimage
in $X$ of $\tau_{\E 1} (X/X_1)$ for the quotient map $X\lra X/X_1$, and $X_3=X$.

\begin{theorem} \label{th}
There is a functorial filtration $0=X_0 \subseteq X_1 \subseteq X_2
\subseteq X_3 = X$ with $X_{i+1}/X_{i}\in \E i$ for $i=0,1,2$.
\end{theorem}
\begin{proof}
First, $\E 0$ is closed under factors, and so $X_1\in \E 0$. Also, $X/X_1
\in KerF^0$ since by taking pullback, any nonzero map $T\lra Z$
induces a map $\E 0$ to $X$ which does not factor through the inclusion
$\tau_{\E 0}X \lra X$. The existence of the filtration then follows from Part a)
of Lemma \ref{keylemma}. We conclude the proof by showing functoriality.
Given $f:X\lra Y$, since trace is a functor, there are induced maps $f:X_1 \lra Y_1$,
and therefore maps $X/X_1 \lra Y/Y_1$, by the functoriality of trace again,
there are maps $X_2/X_1\lra Y_2/Y_1$, since these are induced by $f$,
we must have $f(X_2)\subseteq Y_2$. Hence, the filtration is functorial.
\end{proof}

This theorem completes proof of the existence of the filtration in Theorem \ref{Theo1}.
The uniqueness will be shown at the end of this section.

\begin{lemma}
$KerHom(\E 0,-)=KerF^0$
\end{lemma}
\begin{proof}
$KerHom(\E 0,-)\subseteq KerF^0$ since $T\in \E 0$ and $F^0=Hom(T,-)$.
Assume there is a nonzero map $E\lra Z$ for $E\in \E 0$. We may assume that
$E\rightarrow Z$ is an inclusion, since $\E 0$ is closed under factors. But $E$ has
a non-zero subobject from $FacT$, and therefore there is a nonzero map
$T\rightarrow E \rightarrow Z$. This proves the other inclusion.
\end{proof}

We describe the category $\E 2$ as follows.

\begin{lemma} \label{cate2}
$\E 2=KerF^0\bigcap KerHom(\E 1,-)$
\end{lemma}
\begin{proof}
The inclusion $KerF^0\bigcap KerHom(\E 1,-)\subseteq \E 2$ follows from
part a) of Lemma \ref{keylemma}. Let $X\in \E 2$. Then $X$ has a
filtration with factors isomorphic to subobjects of objects in $\F 2\subseteq
Ker F^0$, and so $X\in KerF^0$. We show that $Hom(\E 1,\K 2)=0$.
Let $Y\in \F 2$. Then $Y\in Ker F^0$ and Lemma \ref{keylemma} a) gives
us an injection $F^1\tau_{\E 1} Y \lra F^1Y$, but then $F^1\tau_{\E 1}Y=0$
since $Y\in Ker F^1$. Thus $\tau_{\E 1}Y=0$ and therefore $Hom(\E 1,Y)=0$.
Any object in $\K 2$ is a subobject of an object in $\F 2$, therefore
$Hom(\E 1, \K 2)=0$.

Now let $X\in \E 2$. Then there is a short exact sequence
$$0 \lra L \lra X \lra K \lra 0$$ $K\in \K 2$ and $L \in \E 2$. Then
any map from $\E 1$ to $X$ must factor through $L$, and so
by induction on the length of a filtration of $X$ we are done.
\end{proof}

We end this section by summarizing some facts about the subcategories $\E i$.
We have already seen that $\E 0$ is closed under factors.

\begin{lemma} The following is true.
\begin{itemize}
\item[a)] $\E 2$ is closed under submodules.
\item[b)] $\E 1$ is closed under images of maps to
$Ker F^0$.
\end{itemize}
\end{lemma}
\begin{proof}
a) follows from Lemma \ref{cate2}.

Let $E\lra Y$ with $E\in \E 1$ and $Y\in Ker F^0$, and let $Z$ be the
image of the map. Then $Z\in KerF^0$, since $KerF^0$ is closed under submodules.
Also $Z\in KerF^2$ since $\E 1\subseteq Ker F^2$, and $E\lra Z$ is surjective.
This proves b).
\end{proof}

We also note the following.

\begin{theorem} If $j>i$, then $Hom(\E i,\E j)=0$. In particular, the
subcategories $\E 0$, $\E 1$ and $\E 2$ are pairwise disjoint.
\end{theorem}
\begin{proof}
We have $\E 1, \E 2\subseteq Ker F^0=KerHom(\E 0,-)$. The other
case follows from Lemma \ref{cate2}. That the categories are
disjoint is then a trivial consequence.
\end{proof}

Using this theorem we may show the following uniqueness property
of the filtration in Theorem \ref{th}. In particular, this shows that
the filtrations in Corollary \ref{cor} and Theorem \ref{th} coincide,
and concludes the proof of Theorem \ref{Theo1}.

\begin{lemma}
If $X$ has a filtration $0=Y_0 \subseteq Y_1 \subseteq Y_2
\subseteq Y_3 = X$ with $Y_{i+1}/Y_{i}\in \E i$ for $i=0,1,2$, then
$Y_i=X_i$ for $i=0,1,2,3$.
\end{lemma}
\begin{proof}
Clearly, $Y_0=X_0$ and $Y_3=X_3$. By the definition of $X_1$ as the
trace of $\E 0$ in $X$, we have $Y_1\subseteq X_1$. If the
inclusion is proper, then there is a non-zero map $E\lra X/Y_1$ from $\E 0$, but
this is impossible since $X/X_1$ is filtered by $\E 1$ and $\E 2$ and
there are no maps from $\E 0$ to $\E 1$ and $\E 2$ by the previous Lemma.
Thus $Y_1=X_1$. Similarly, $Y_2=X_2$. This finishes the proof.
\end{proof}

\section{The extension closure of $FacT$} \label{ext}

In this section let $\catA=modA$ for a finite dimensional algebra $A$.
For any $X\in \catA$ there is a natural map $$\phi_X:J^0_0X=T\otimes_B
Hom_\catA(T,X)\lra X$$ given by $t\otimes f\mapsto f(t)$, with image
equal to the trace $\tau_TX$ of $T$ in $X$. From Lemma \ref{seqlemma}
there is the short exact sequence $$0 \lra J^1_2X\lra J^0_0X
\stackrel{\phi}{\lra} X$$ with image $im\phi=Z_1$ equal to the first
term $Z_1$ in the filtration of Theorem \ref{Theo2}.

\begin{lemma}
$Z_1=\tau_TX$.
\end{lemma}
\begin{proof}
If $X\in FacT$, then $X\in \K 0$ by Lemma \ref{FacT}, and so there is a
short exact sequence $$0 \lra X_2 \lra X_0 \lra X \lra 0$$ for $X_2\in \F 2$
and $X_0\in \F 0$. But then by Lemma \ref{zerolemma}, $X_0\cong J^0_0X$ and
$X_2\cong J^1_2X$, and so $\phi$ is surjective by comparing dimensions
and therefore $Z_1=X=\tau_TX$.

Now, for arbitrary $X$, let $i:\tau_TX\lra X$ be the inclusion map. There is a commutative
diagram $$\xymatrix{J^0_0\tau_TX \ar[r]^{J^0_0i} \ar[d]^{\phi_{\tau_TX}}
& J^0_0X \ar[d]^{\phi_X} \\ \tau_TX \ar[r]^i & X}.$$ Since $\phi_{\tau_TX}$ is
surjective, we have $\tau_TX\subseteq im\phi_X$. Then since $J^0_0X\in FacT$
and therefore $im \phi_X\in FacT$ we have $\tau_TX=im\phi_X=Z_1$.
\end{proof}

We give an example showing that $\K 0=FacT\subsetneq \E 0$. In particular,
this example shows that the filtration of Theorem \ref{Theo2} can be more refined than
the filtration of Theorem \ref{Theo1}. Consider the
following quiver. $$\Delta:\xymatrix{1 \ar@/^/[r] \ar@/_/[r] & 2 \ar[r] & 3}$$
Let $A=k\Delta/rad(k\Delta)^2$ and let $T=DA$. Then $FacT$ consists of direct
sums of direct summands of $T\oplus S_2$, where $S_2$ is the simple at vertex $2$.
The representation $$\xymatrix{k \ar@/^/[r]^1 \ar@/_/[r]_1 & k \ar[r] & 0}$$ is an
extension of $S_1\in Fac T$ by $S_2$,  but is clearly not in $FacT$. Hence
$FacT$ is not closed under extensions and so $FacT\subsetneq \E 0$.

However, we recall the following positive result. An indecomposable
injective $A$-module $I$ is called maximal if any surjective map
$J\lra I$ with $J$ injective is split.
The following proposition follows from the argument preceding
Proposition 6.9 of \cite{AS}.

\begin{proposition} \cite{AS} \label{ASLemma}
$FacDA$ is closed under extensions if and only if any maximal injective module
has projective dimension at most one.
\end{proposition}

Let $A$ be a Nakayama algebra. That is, $A=kQ/I$,
where $Q$ is an oriented cycle, or a quiver of type $\mathbb{A}$
with linear orientation.

\begin{proposition}
Let $A$ be a Nakayama algebra. Then $FacDA$ is closed under extensions.
\end{proposition}
\begin{proof}
A maximal injective $A$-module is projective. The proposition then
follows from Proposition \ref{ASLemma}.
\end{proof}

Finailly, we given an example showing that filtrations with factors in $\F i$ do
not exist in general. Let $$\Delta:\xymatrix{1 \ar[r] & 2 \ar[r] & 3 \ar[r] & 4 }$$
and let $A$ be the path algebra of $\Delta$ with a relation equal to the path from
vertex $2$ to $4$. Then $A$ has global dimension two. Let $T=DA$.
Then $S_4\in \F 2$, $I_4\in \F 0$ and $S_3$ is the cokernel of the inclusion
$S_4\longrightarrow I_4$, and so $S_3\in \K 0$, but not in $\F 0$.

\bibliographystyle{annotate}

{\parindent=0cm
Bernt Tore Jensen \\
Institut de mathematique de Jussieu, \\
UMR 7586, \\
175 rue Chevaleret, \\
75013, Paris,\\
France\\\\}

{\parindent=0cm
Dag Madsen \\
Department of Mathematics, \\
215 Carnegie,
Syracuse University,\\
Syracuse, NY 13244-1150,\\
USA \\\\}

{\parindent=0cm
Xiuping Su \\
Mathematical Sciences, \\
University of Bath,\\
Bath BA2 7JY,\\
United Kingdom.\\
email: xs214@bath.ac.uk }

\end{document}